\documentclass[12pt,a4paper,reqno,twoside]{amsart}

\usepackage[english]{babel}
\usepackage{stmaryrd}
\usepackage{dsfont}
\usepackage[symbol*,ragged]{footmisc}
\usepackage[colorlinks,linkcolor=red,anchorcolor=blue,citecolor=blue,urlcolor=blue]{hyperref}
\usepackage{color,xcolor}

\usepackage{geometry}
\usepackage{amssymb}
\usepackage{amsmath}
\usepackage{mathrsfs}
\usepackage{amsfonts}
\usepackage{epsfig}

\usepackage{amsthm}
\usepackage{amsxtra}
\usepackage{bbding}
\usepackage{epsfig}
\usepackage{graphicx}
\usepackage{latexsym}
\usepackage{mathbbol}
\usepackage{bbold}

\usepackage{pifont}
\usepackage{wasysym}
\usepackage{skull}
\usepackage{float}

\DeclareSymbolFontAlphabet{\mathbb}{AMSb}
\DeclareSymbolFontAlphabet{\mathbbol}{bbold}

\usepackage{amscd}
\usepackage[all]{xy}
\allowdisplaybreaks[4]
\usepackage{setspace}

\theoremstyle{plain}
\newtheorem{theorem}{Theorem}[section]
\newtheorem{proposition}{Proposition}[section]
\newtheorem{lemma}[proposition]{Lemma}
\newtheorem{corollary}[theorem]{Corollary}
\newtheorem*{corollary*}{Corollary}

\theoremstyle{remark}

\numberwithin{equation}{section}
\addtocounter{footnote}{1}

\renewcommand{\footnoterule}{
  \kern -3pt
  \hrule width 2.5in height 0.4pt
  \kern 3pt
}

\makeatletter
\@ifundefined{MakeUppercase}{}{}
\makeatother

\begin{document}
	
\title[Vinogradov's three primes theorem in the intersection of multiple P--S sets ]
	  {Vinogradov's three primes theorem in the intersection of multiple Piatetski--Shapiro sets}

\author[Xiaotian Li, Jinjiang Li, Min Zhang]{Xiaotian Li \quad \& \quad Jinjiang Li \quad \& \quad Min Zhang}

\address{Department of Mathematics, China University of Mining and Technology,
         Beijing, 100083, People's Republic of China}

\email{xiaotian.li.math@gmail.com}


\address{(Corresponding author) Department of Mathematics, China University of Mining and Technology,
         Beijing 100083, People's Republic of China}

\email{jinjiang.li.math@gmail.com}

\address{School of Applied Science, Beijing Information Science and Technology University,
		 Beijing 100192, People's Republic of China  }

\email{min.zhang.math@gmail.com}

\date{}

\footnotetext[1]{Jinjiang Li is the corresponding author.  

\quad\,\,
{\textbf{MR(2020) Subject Classification}}: 11N05, 11N80, 11L07, 11L20, 11P32

}

\begin{abstract}
Vinogradov's three primes theorem indicates that, for every sufficiently large odd integer $N$, the equation $N=p_1+p_2+p_3$ is solvable in prime variables $p_1,p_2,p_3$. In this paper, it is proved that Vinogradov's three primes theorem still holds with three prime variables constrained in the intersection of multiple Piatetski--Shapiro sequences.
\end{abstract}

\maketitle

\section{Introduction and Main Result}

In 1937, Vinogradov \cite{Vinogradov-1937} proved that, for every sufficiently large odd integer $N$, the equation
\begin{equation}\label{1.1}
 N=p_1+p_2+p_3
\end{equation}
is solvable in prime variables $p_1,p_2,p_3$. More precisely, he solved asymptotic form as follows
\begin{equation}\label{1.2}
 \sum_{N=p_1+p_2+p_3}1=\frac{1}{2}\mathfrak{S}(N)\frac{N^2}{\log^3{N}}+O\left(\frac{N^2}{\log^4 N}\right),
\end{equation}
where $\mathfrak{S}(N)$ is the singular series defined by
\begin{equation}\label{1.3}
\mathfrak{S}(N)=\prod_{p|N}\left(1-\frac{1}{(p-1)^2}\right)\prod_{p\nmid N}\left({1+\frac{1}{(p-1)^3}}\right).
\end{equation}
The asymptotic formula (\ref{1.2}) is called three primes theorem, or Goldbach--Vinogradov theorem.
Recently, Helfgott \cite{Helfgott-2013-1,Helfgott-2013-2,Helfgott-arXiv,Helfgott-2014,Helfgott-2015} completely solved the problem and proved that the ternary Goldbach conjecture is true.

Let $\gamma\in(\frac{1}{2},1)$ be a fixed real number. The Piatetski--Shapiro sequences are sequences of the form
\begin{equation*}
 \mathscr{N}_{\gamma}:=\big\{\lfloor n^{1/\gamma}\rfloor:\,n\in \mathbb{N}^+\big\}.
\end{equation*}
Such sequences have been named in honor of Piatetski--Shapiro, who \cite{Piatetski-Shapiro-1953}, in
1953, proved that $\mathscr{N}_{\gamma}$ contains infinitely many primes provided that $\gamma\in(\frac{11}{12},1)$.
The prime numbers of the form $p=\lfloor n^{1/\gamma}\rfloor$ are called \textit{Piatetski--Shapiro primes of type $\gamma$}. More precisely, for such $\gamma$ Piatetski--Shapiro \cite{Piatetski-Shapiro-1953} showed that the counting function
\begin{equation*}
 \pi_\gamma(x):=\#\big\{\textrm{prime}\,\, p\leqslant x:\,p=\lfloor n^{1/\gamma}\rfloor\,\,\textrm{for some}\,\,
 n\in\mathbb{N}^+ \big\}
\end{equation*}
satisfies the asymptotic property
\begin{equation*}
\pi_{\gamma}(x)=\frac{x^{\gamma}}{\log x}(1+o(1))
\end{equation*}
as $x\to\infty$. Since then, the range for $\gamma$ of the above asymptotic formula in which it is known that $\mathscr{N}_{\gamma}$ contains infinitely many primes has been enlarged many times (see the literatures \cite{Kolesnik-1967,Leitmann-1975,Leitmann-1980,Heath-Brown-1983,Kolesnik-1985,Liu-Rivat-1992,
Rivat-1992,Rivat-Sargos-2001}) over the years and is currently known to hold for all $\gamma\in(\frac{2426}{2817},1)$ thanks to Rivat and Sargos \cite{Rivat-Sargos-2001}. Rivat and Wu \cite{Rivat-Wu-2001} also showed that there exist infinitely many Piatetski--Shapiro primes for $\gamma\in(\frac{205}{243},1)$ by showing a lower bound of $\pi_\gamma(x)$ with the expected order of magnitude. We remark that if $\gamma>1$ then $\mathscr{N}_\gamma$ contains all natural numbers, and hence all primes, particularly.

In 1992, Balog and Friedlander \cite{Balog-Friedlander-1992} proved that the equation (\ref{1.1}) is solvable with prime variables constrained to the Piatetski--Shapiro sequences. To be specific, they proved that, for $1/2<\gamma_j<1\,\,(j=1,2,3)$ with
\begin{align*}
& 9(1-\gamma_3)<1,\quad
  9(1-\gamma_2)+6(1-\gamma_3)<1,\\
& 9(1-\gamma_1)+6(1-\gamma_2)+6(1-\gamma_3)<1,
\end{align*}
the asymptotic formula
\begin{equation}\label{1.5}
\frac{1}{\gamma_1\gamma_2\gamma_3}\sum_{\substack{N=p_1+p_2+p_3\\p_j=\lfloor n^{1/\gamma_j}\rfloor}}\prod_{j=1}^3{p_j}^{1-\gamma_j}\log p_j
=\frac{1}{2}\mathfrak{S}(N)N^2+O_A\left(\frac{N^2}{\log^A N}\right)
\end{equation}
holds, where  $A>0$ is an arbitrary real number, and $\mathfrak{S}(N)$ is defined as in (\ref{1.3}). An corollary of (\ref{1.5}) implies that, for $8/9<\gamma<1$, every sufficiently large odd integer $N$ can be represented as the sum of two primes and a Piatetski--Shapiro prime of type $\gamma$. Afterwards, in 1997, Kumchev \cite{Kumchev-1997} improved the result of Balog and Friedlander \cite{Balog-Friedlander-1992}.

In 1982, Leitmann \cite{Leitmann-1982} proved that $\mathscr{N}_{\gamma_1}\cap\mathscr{N}_{\gamma_2}$ contains infinitely many primes provided that $55/28<\gamma_1+\gamma_2<2$. More precisely, Leitmann \cite{Leitmann-1982} showed that, for $55/28<\gamma_1+\gamma_2<2$, the counting function
\begin{equation*}
\pi(x;\gamma_1,\gamma_2):=\#\big\{\textrm{prime}\,\, p\leqslant x:\,p=\lfloor n_1^{1/\gamma_1}\rfloor=
\lfloor n_2^{1/\gamma_2}\rfloor\,\,\textrm{for some}\,\,
 n_1,n_2\in\mathbb{N}^+ \big\}
\end{equation*}
satisfies the asymptotic property
\begin{equation}\label{2-P-S-PNT}
\pi(x;\gamma_1,\gamma_2)
=\frac{\gamma_1\gamma_2}{\gamma_1+\gamma_2-1}\cdot\frac{x^{\gamma_1+\gamma_2-1}}{\log x}(1+o(1))
\end{equation}
as $x\to\infty$. In 1983, Sirota \cite{Sirota-1983} proved that (\ref{2-P-S-PNT}) holds for $31/16<\gamma_1+\gamma_2<2$. In 2014, Baker \cite{Baker-2014} established that (\ref{2-P-S-PNT}) holds for  $23/12<\gamma_1+\gamma_2<2$. The hitherto best result in this direction is due to Li, Zhai and Li \cite{Li-Zhai-Li-2023}, who showed that (\ref{2-P-S-PNT}) holds for $21/11<\gamma_1+\gamma_2<2$.

For any fixed integer $k\geqslant2$, we define
\begin{equation}\label{define-symbol}
\mathcal{C}^{(i)}_k=\frac{1}{\gamma^{(i)}_1\cdots\gamma^{(i)}_k},\quad \sigma^{(i)}_k=k-\big(\gamma^{(i)}_1+\cdots+\gamma^{(i)}_k\big),\quad (i=1,2,3;\,\,j=1,\dots,k).
\end{equation}
\noindent
In 2022, Li and Zhai \cite{Li-Zhai-2022} proved that the equation (\ref{1.1}) is solvable
with $p_1,p_2,p_3$ constrained in the intersection of two Piatetski--Shapiro sequences. To be specific, they showed that, for fixed real numbers $1/2<\gamma^{(i)}_2<\gamma^{(i)}_1<1\,\,(i=1,2,3)$ subject to the conditions
\begin{align*}
 12\cdot\sigma_2^{(3)}<1,\quad 12\cdot\sigma_2^{(2)}+19\cdot\sigma_2^{(3)}<1,\quad
 12\cdot\sigma_2^{(1)}+19\cdot\sigma_2^{(2)}+19\cdot\sigma_2^{(3)}<1,
\end{align*}
the asymptotic formula
\begin{equation*}
\mathcal{C}^{(1)}_2\mathcal{C}^{(2)}_2\mathcal{C}^{(3)}_2
      \sum_{\substack{N=p_1+p_2+p_3\\p_i=\lfloor n^{1/\gamma^{(i)}_1}_1\rfloor=\lfloor n^{1/\gamma^{(i)}_2}_2\rfloor}}
      \prod_{i=1}^3{p_i}^{\sigma^{(i)}_2}
      =\frac{1}{2}\mathfrak{S}(N)\frac{N^2}{\log^3 N}+O\left(\frac{N^2}{\log^4N}\right)
\end{equation*}
holds, where $\mathfrak{S}(N)$ is defined as in (\ref{1.3}). In 1999, Zhai \cite{Zhai-1999}, who investigated $k$--dimensional $(k\geqslant 3)$ Piatetski--Shapiro prime number theorem firstly, obtained that, for $1/2<\gamma_k<\dots<\gamma_1<1$ such that $\sigma_k<k/(4k^2+2)$, the asymptotic formula
\begin{equation}\label{k-P-S-PNT}
\pi(x;\gamma_1,\dots,\gamma_k)=\frac{\gamma_1\cdots\gamma_k}{1-\sigma_k}\cdot\frac{x^{1-\sigma_k}}{\log x}(1+o(1))
\end{equation}
holds, where $\pi(x;\gamma_1,\dots,\gamma_k)$ denotes the counting function
\begin{equation*}
\pi(x;\gamma_1,\dots,\gamma_k):=\#\big\{\textrm{prime}\,\, p\leqslant x:\,p=\lfloor n_1^{1/\gamma_1}\rfloor=\cdots=
\lfloor n_k^{1/\gamma_k}\rfloor\,\,\textrm{for some}\,\,
 n_1,\dots,n_k\in\mathbb{N}^+ \big\}.
\end{equation*}
Later, in 2014, Baker \cite{Baker-2014} enhanced the result of Zhai \cite{Zhai-1999} and showed that (\ref{k-P-S-PNT}) holds for $\sigma_k<1/3k$ for $k\geqslant6$.

In this paper, we shall investigate the solvability of (\ref{1.1}) with three prime variables
constrained in the intersection of multiple Piatetski--Shapiro sequences, and establish the following theorem.
\begin{theorem}\label{Theorem 1.1}
For $k\geqslant 3,\,i\in\{1,2,3\}$, let $1/2<\gamma^{(i)}_k<\dots<\gamma^{(i)}_1\leqslant1$ be fixed real numbers such that
\begin{align*}
 & 4k\cdot\sigma^{(3)}_k<1-\varpi_k,\quad4k\cdot\sigma^{(2)}_k+2k(k+1)\cdot\sigma^{(3)}_k<1-\varpi_k,\\
 & 4k\cdot\sigma^{(1)}_k+2k(k+1)\cdot\sigma^{(2)}_k+2k(k+1)\cdot\sigma^{(3)}_k<1-\varpi_k,
\end{align*}
where
\begin{equation}\label{varpi-k}
\varpi_3=\frac{1}{3}\cdot\frac{1}{12},
\quad\varpi_4=\frac{1}{4}\cdot\frac{1}{16},
\quad\varpi_5=\frac{1}{5}\cdot\frac{1}{18},
\quad\varpi_{k}=\frac{1}{k}\cdot\frac{1}{3k}\,\,(k\geqslant6).
\end{equation}
Then there holds the asymptotic formula
\begin{equation}\label{asymptotic} \mathcal{C}^{(1)}_k\mathcal{C}^{(2)}_k\mathcal{C}^{(3)}_k\sum_{\substack{N=p_1+p_2+p_3\\p_i=\lfloor n^{1/\gamma^{(i)}_1}_1\rfloor=\cdots=\lfloor n^{1/\gamma^{(i)}_k}_k\rfloor}}\prod_{i=1}^3p^{\sigma^{(i)}_k}_i
=\frac{1}{2}\mathfrak{S}(N)\frac{N^2}{\log^3 N}+O\left(\frac{N^2}{\log^4 N}\right),
\end{equation}
where $\mathfrak{S}(N)$ is defined as in (\ref{1.3}). In particular, when $k=3$, (\ref{asymptotic}) holds, provided that
\begin{align*}
 & 12\cdot\sigma^{(3)}_3<1-1/24,\quad12\cdot\sigma^{(2)}_3+26\cdot\sigma^{(3)}_3<1-1/24,\\
 & 12\cdot\sigma^{(1)}_3+26\cdot\sigma^{(2)}_3+26\cdot\sigma^{(3)}_3<1-1/24.
\end{align*}
\end{theorem}

Taking $\gamma^{(1)}_j=\gamma^{(2)}_j=\gamma^{(3)}_j=\gamma_j\,(j=1,\dots, k)$ in Theorem \ref{Theorem 1.1}, we obtain the following Corollary \ref{Corollary-2}.
\begin{corollary}\label{Corollary-2}
For $k\geqslant3$, let $1/2<\gamma_k<\dots<\gamma_1<1$ be fixed real numbers satisfying $\sigma_3<\frac{1-1/24}{64}$ and $\sigma_k<\frac{1-\varpi_k}{4k^2+8k}\,(k\geqslant4)$. Then every sufficiently large odd integer $N$ can be represented as a sum of three primes with three primes constrained in the intersection $\mathscr{N}_{\gamma_1}\cap\cdots\cap\mathscr{N}_{\gamma_k}$.
\end{corollary}

By choosing $\gamma^{(1)}_{j}=\gamma^{(2)}_{j}=1\,(j=1,\dots,k)$ in Theorem \ref{Theorem 1.1}, we get Corollary \ref{Corollary-3}.
\begin{corollary}\label{Corollary-3}
For $k\geqslant3$, let $1/2<\gamma^{(3)}_k<\dots<\gamma^{(3)}_1<1$ be fixed real numbers satisfying $\sigma^{(3)}_3<\frac{1-1/24}{12}$ and $\sigma^{(3)}_k<\frac{1-\varpi_k}{4k}\,(k\geqslant4)$. Then every sufficiently large odd integer $N$ can be represented as a sum of three primes with one prime constrained in the intersection  $\mathscr{N}_{\gamma_1}\cap\cdots\cap\mathscr{N}_{\gamma_k}$.
\end{corollary}

\smallskip
\textbf{Notation.}
Throughout this paper, $\varepsilon$ is sufficiently small positive number, which may be different in each occurrences. Let $p$, with or without subscripts, always denote a prime number. We use $\lfloor t\rfloor$ and $\|t\|$ to denote the integral part of $t$ and the distance from $t$ to the nearest integer, respectively. As usual, $\Lambda(n)$, $\mu(n)$ and $d(n)$ denote von Mangoldt's function, M\"{o}bius' function and Dirichlet divisor function, respectively. We write $\psi(t)=t-\lfloor t\rfloor-1/2,e(x)=e^{2\pi ix}$. The notation $f(x)\ll g(x)$ means  $f(x)=O(g(x))$. We use $\mathbb{N}^+$ and $\mathbb{Z}$ to denote the set of positive natural numbers and the set of integers, respectively.

\section{Preliminaries}

\begin{lemma}\label{2-3-derivative}
Suppose that $5<A<B\leqslant 2A$, and $f''(x)$ is continuous on $[A,B]$. If $0<r_1\lambda_1\leqslant |f''(x)|\leqslant r_2\lambda_1$, then
\begin{equation}
\sum_{A<n\leqslant B}e(f(n))\ll A\lambda^{1/2}_1+\lambda_1^{-1/2}.\label{2-derivative}
\end{equation}
If $0<r_3\lambda_2\leqslant |f'''(x)|\leqslant r_4\lambda_2$, then
\begin{equation}
\sum_{A<n\leqslant B}e(f(n))\ll A\lambda^{1/6}_2+\lambda_2^{-1/3}.\label{3-derivative}
\end{equation}
where $r_i\,(i=1,2,3,4)$ are absolute constants.
\end{lemma}
\begin{proof}
For (\ref{2-derivative}), one can see Corollary 8.13 of Iwaniec and Kowalski \cite{Iwaniec-Kowalski-2004}, or Theorem 5 of Chapter 1 in Karatsuba \cite{Karatsuba-1993}. For (\ref{3-derivative}), one can see Corollary 4.2 of Sargos \cite{Sargos-1995}.
\end{proof}

\begin{lemma}\label{Zhai-1999-Th1}
Suppose that $k\geqslant 3$ and $1/2<\gamma_k<\dots <\gamma_1<1$ such that
\begin{equation*}
\gamma_1+\dots+\gamma_k>k-\frac{k}{4k^2+2}.
\end{equation*}
Then we have
\begin{equation*}
\pi(x;\gamma_1,\dots,\gamma_k)=\gamma_1\cdots\gamma_k\int_{2}^x\frac{t^{\gamma_1+\dots +\gamma_k-k}}{\log t}\mathrm{d}t+O\left(x^{\gamma_1+\dots+\gamma_k-k+1}e^{-c_0\sqrt{\log x}}\right),
\end{equation*}
where $c_0>0$ is an absolute constant.
\end{lemma}
\begin{proof}
See Theorem 1 of Zhai \cite{Zhai-1999}.
\end{proof}

\begin{lemma}\label{Balog-Friedlander-Th4}
Suppose that $0<\gamma<1,\,\, 0\leqslant \delta \leqslant 1-\gamma$, which satisfy $9(1-\gamma)+12\delta<1$.
Then, uniformly for $\alpha \in (0,1)$, we have
\begin{equation*}
\frac{1}{\gamma}\sum_{p\leqslant N}e(\alpha p)p^{1-\gamma}(\psi((p+1)^\gamma)-\psi(p^\gamma))
     \ll N^{1-\delta-\varepsilon},
\end{equation*}
where the implied constant may only depend on $\gamma$ and $\delta$.
\end{lemma}
\begin{proof}
See Theorem 4 of Balog and Friedlander \cite{Balog-Friedlander-1992}.
\end{proof}

\begin{lemma}\label{Li-Zhai-2022}
Suppose that $1/2<\gamma_2<\gamma_1<1,\,\,0\leqslant \delta <1/2$ such that
$12(2-(\gamma_1+\gamma_2))+38\delta<1$. Then, uniformly for $\alpha \in (0,1)$, we have
\begin{align*}
\sum_{p\leqslant N}e(\alpha p)p^{2-(\gamma_1+\gamma_2)}
\left(\psi(-(p+1)^{\gamma_1})-\psi(-p^{\gamma_1})\right)
\left(\psi(-(p+1)^{\gamma_2})-\psi(-p^{\gamma_2})\right) \ll N^{1-\delta-\varepsilon}.
\end{align*}
where the implied constant may depend at most on $\gamma_1,\gamma_2$, and $\delta$ only.
\end{lemma}
\begin{proof}
See Theorem 4 of Li and Zhai \cite{Li-Zhai-2022}.
\end{proof}

\begin{lemma}\label{Finite-Fourier-expansion}
  For any $H>1$, we have
\begin{equation*}
  \psi(\theta)=-\sum_{1\leqslant |h|\leqslant H}\frac{e(h\theta)}{2\pi ih}+O\bigg(\min\bigg(1,\frac{1}{H\|\theta\|}\bigg)\bigg),
\end{equation*}
\begin{equation*}
  \min\bigg(1,\frac{1}{H\|\theta\|}\bigg)=\sum_{h=-\infty}^{\infty}a(h)e(h\theta),
\end{equation*}
where
\begin{equation*}
a(0)\ll \frac{\log 2H}{H},\quad a(h)\ll \min\Bigg(\frac{1}{|h|},\frac{H}{h^2}\Bigg)\,\,\textrm{with}\,\, h\not=0.
\end{equation*}
\end{lemma}
\begin{proof}
See the arguments on page 245 of Heath--Brown \cite{Heath-Brown-1983}.
\end{proof}

\begin{lemma}\label{Weyl-Inequality}
Let $z(n)$ be any complex numbers. Then, for $1\leqslant Q\leqslant N$, there holds
\begin{equation*}
\left|\sum_{N<n\leqslant CN}z(n)\right|^2\ll \frac{N}{Q}\sum_{0\leqslant |q|\leqslant Q}\left(1-\frac{q}{Q}\right)\mathfrak{R}\sum_{N<n\leqslant CN-q}z(n)\overline{z(n+q)}.
\end{equation*}
\end{lemma}
\begin{proof}
See Lemma 2.5 of Graham and Kolesnik \cite{Graham-Kolesnik-1991}.
\end{proof}

\begin{lemma}\label{H-B-Identity}
Let $z\geqslant1$ and $\nu\geqslant1$. Then, for any $n\leqslant2z^{\nu}$, there holds
\begin{equation*}
\Lambda(n)=\sum_{j=1}^\nu(-1)^{j-1}\binom{\nu}{j}\mathop{\sum\cdots\sum}\limits_{\substack{n_1n_2\cdots n_{2j}=n\\n_{j+1,\dots,n_{2j}\leqslant z}}}(\log n_1)\mu(n_{j+1})\cdots\mu(n_{2j}).
\end{equation*}
\end{lemma}
\begin{proof}
See the arguments on pp. 1366--1367 of Heath--Brown \cite{Heath-Brown-1982}.
\end{proof}

\begin{lemma}\label{Zhai-1999-prop}
Let $s\geqslant 2$ be a fixed integer, $\Delta>0$. Suppose that $a_1,\dots,a_s$ are real numbers with $a_1\cdots a_s\neq 0$. Let $\alpha_1,\dots,\alpha_s$ be distinct real constants such that $\alpha_j \notin\mathbb{Z}\,\,(j=1,\dots,s)$. Let $I$ denote a subinterval of $[1,2]$ such that
\begin{equation}\label{Zhai-1999-Prop1-condition}
|a_1t^{\alpha_1}+a_2t^{\alpha_2}+\cdots+a_st^{\alpha_s}|\leqslant\Delta
\end{equation}
for any $t\in I$. Then one has
\begin{align*}
|I|\ll \left(\frac{\Delta}{|a_1|+\cdots+|a_s|}\right)^{\frac{1}{s-1}}.
\end{align*}
where $|I|$ denotes the cardinality of $I$. Further more let $I_1,\dots, I_{s^*}$ denote all the disjoint subintervals of $[1,2]$, which make (\ref{Zhai-1999-Prop1-condition}) hold, then $s^*=O(1)$.
\end{lemma}
\begin{proof}
See Proposition 1 of Zhai \cite{Zhai-1999}.  $\hfill$
\end{proof}

\begin{lemma}\label{Zhai-1999-prop-4}
Suppose that $s\geqslant 2,\,\,1/2<\gamma_s<\dots<\gamma_1<1$ are fixed real numbers. Define
\begin{equation*}
S^*(M;\gamma_1,\dots,\gamma_{s}):=\sup_{(u_1,\dots,u_s)\in [0,1]^{s}}\sum_{M<m\leqslant2M}\prod_{j=1}^{s}\min\left(1,\frac{1}{H_j\|(m+u_j)^{\gamma_j}\|}\right),
\end{equation*}
where $H_j>1(j=1,\dots,s)$ are real numbers. If $\gamma_1+\dots+\gamma_{s}>s-1/(s+1)$, then
\begin{equation*}
S^*(M;\gamma_1,\dots,\gamma_{s})\ll M(H_1\cdots H_{s})^{-1}(\log H)^{s}+H^{\frac{s}{s+1}}(\log H)^{s}.
\end{equation*}
\end{lemma}
\begin{proof}
See Proposition 4 of Zhai \cite{Zhai-1999}.
\end{proof}

\section{Estimate of Exponential Sums}
For $k\geqslant3$, suppose that $a_1,a_2,\dots,a_k$ are real numbers with $a_1a_2\cdots a_k\neq0$. Let $\gamma_1,\gamma_2,\dots,\gamma_k$ be real numbers such that $\gamma_j\neq\mathbb{Z}\,(1\leqslant j\leqslant k)$ and $\alpha\in [0,1]$ be a real number. $M$ and $M_1$ are sufficiently large numbers such that $M<M_1\leqslant2M$. Define
\begin{equation*}
S_k(M;\gamma_1,\dots,\gamma_k):=\sum_{M<m\leqslant M_1}e(\alpha m+a_1m^{\gamma_1}+\dots+a_km^{\gamma_k}).
\end{equation*}
For convenience, we always set $R:=|a_1|M^{\gamma_1}+\dots+|a_k|M^{\gamma_k}$. By Lemma \ref{2-derivative}, we have the following proposition.

\begin{proposition}\label{prop-Sk(M-gamma-alpha)}
For fixed integer $k\geqslant3$, there holds uniformly for $\alpha\in[0,1]$ that
\begin{equation}\label{prop-Sk(M-gamma-alpha)-2}
S_k(M;\gamma_1,\dots,\gamma_k)\ll R^{1/2}+MR^{-\frac{1}{k+1}},
\end{equation}
and
\begin{equation}\label{prop-Sk(M-gamma-alpha)-3}
S_k(M;\gamma_1,\dots,\gamma_k)\ll M^{1/2}R^{1/6}+MR^{-\frac{1}{k+2}}.
\end{equation}
\end{proposition}
\begin{proof}
Let $f_k(m)=\alpha m+a_1m^{\gamma_1}+\dots+a_km^{\gamma_k}$. Then
\begin{align*}
 & f_k''(m)=a_1\gamma_1(\gamma_1-1)m^{\gamma_1-2}+\dots+a_k\gamma_k(\gamma_k-1)m^{\gamma_k-2},\\
 & f_k'''(m)= a_1\gamma_1(\gamma_1-1)(\gamma_1-2)m^{\gamma_1-3}+\dots+a_k\gamma_k(\gamma_k-1)(\gamma_k-2)m^{\gamma_k-3}.
\end{align*}

\noindent
\textbf{Case 1.} Let $\Delta_0=o(RM^{-2})$ be a parameter which will be chosen later. Define
\begin{align*}
 & I_0=\{m;m\in (M,M_1], 0<\left|f''_k(m)\right|\leqslant \Delta_0\},\\
 & I_t=\left\{m:m\in (M,M_1],2^{t-1}\Delta_0<\left|f''_k(m)\right|\leqslant 2^t\Delta_0\leqslant \frac{R}{M^2}\right\},\quad
  1\leqslant t\leqslant T:=\left\lfloor\frac{\log \frac{R}{M^2\Delta_0}}{\log 2}\right\rfloor.
\end{align*}
If $n\in I_0$, by the definition of $I_0$, we have
\begin{equation*}
\left|f''_k(m)\right|=\left|a_1\gamma_1(\gamma_1-1)m^{\gamma_1-2}+\dots
+a_k\gamma_k(\gamma_k-1)m^{\gamma_k-2}\right|
\leqslant \Delta_0.
\end{equation*}
It is easy to see that
\begin{equation*}
\left|f''_k(m)\right|=\left|\sum_{j=1}^ka_j\gamma_j(\gamma_j-1)M^{\gamma_1}m^{\gamma_1}_0\right|\ll \Delta_0M^2,
\end{equation*}
where $m_0=m/M$, $m_0\in [1,2]$. By Lemma \ref{Zhai-1999-prop}, we derive that
\begin{align}\label{estimate-I0}
          |I_0|
\ll &\,\, M \left(\frac{\Delta_0M^2}{|a_1\gamma_1(\gamma_1-1)|M^{\gamma_1}+\dots
          +|a_k\gamma_k(\gamma_k-1)|M^{\gamma_k}}\right)^{\frac{1}{k-1}}\nonumber\\
\ll&\,\,  M \left(\Delta_0M^2R^{-1}\right)^{\frac{1}{k-1}}
         \ll M^{1+\frac{2}{k-1}}\Delta^{\frac{1}{k-1}}_0R^{-\frac{1}{k-1}}.
\end{align}
From (\ref{2-derivative}) of Lemma \ref{2-3-derivative} and (\ref{estimate-I0}), by taking $\Delta_0=M^{-2}R^{2/(k+1)}$, we obtain
\begin{align*}
          S_k(M;\gamma_1,\dots,\gamma_k)
   = & \,\, \sum_{m\in I_0}e\left(f_k(m)\right)+\sum_{1\leqslant t\leqslant T}\sum_{m\in I_t}e\left(f_k(m)\right)
                 \nonumber \\
 \ll & \,\, M^{1+\frac{2}{k-1}}\Delta^{\frac{1}{k-1}}_0R^{-\frac{1}{k-1}}+\sum_{1\leqslant t\leqslant T}
            \left(M(2^t\Delta_0)^{1/2}+(2^t\Delta_0)^{-1/2}\right)
                 \nonumber \\
 \ll & \,\, M^{1+\frac{2}{k-1}}\Delta^{\frac{1}{k-1}}_0R^{-\frac{1}{k-1}}
           +M\Delta^{1/2}_0 M^{-1}R^{1/2}\Delta^{-1/2}_0 +\Delta^{-1/2}_0
                 \nonumber \\
 \ll & \,\, MR^{-\frac{1}{k+1}}+R^{1/2},
\end{align*}
which gives the estimate (\ref{prop-Sk(M-gamma-alpha)-2}) of Proposition \ref{prop-Sk(M-gamma-alpha)}.

\noindent
\textbf{Case 2.}
Let $\Delta^*_0=o(RM^{-3})$ be a parameter which will be chosen later. Define
\begin{align*}
 & \mathcal{I}_0=\{m;m\in (M,M_1], 0<\left|f'''_k(m)\right|\leqslant \Delta^*_0\},\\
 & \mathcal{I}_{t^*}=\left\{m:m\in (M,M_1],2^{t^*-1}\Delta_0<\left|f'''_k(m)\right|\leqslant 2^{t^*}\Delta_0\leqslant \frac{R}{M^3}\right\},\quad
  1\leqslant t^*\leqslant T^*:=\left\lfloor\frac{\log \frac{R}{M^3\Delta^*_0}}{\log 2}\right\rfloor.
\end{align*}
If $n\in \mathcal{I}_0$, from the definition of $\mathcal{I}_0$, we get
\begin{equation*}
\left|f'''_k(m)\right|=\left|\sum_{j=1}^ka_j\gamma_j(\gamma_j-1)(\gamma_j-2)m^{\gamma_j-3}\right|
\leqslant \Delta^*_0.
\end{equation*}
It is easy to see that
\begin{equation*}
\left|f'''_k(m)\right|=\left|\sum_{j=1}^ka_j\gamma_j(\gamma_j-1)(\gamma_j-2)M^{\gamma_j}m^{\gamma_j}_0\right|\ll \Delta^*_0M^3,
\end{equation*}
where $m_0=m/M$, $m_0\in [1,2]$. By Lemma \ref{Zhai-1999-prop}, we have
\begin{align}\label{etimate-mathcal(I)0}
           |\mathcal{I}_0|
\ll & \,\, M \left(\frac{\Delta^*_0M^3}{|a_1\gamma_1(\gamma_1-1)(\gamma_1-2)|M^{\gamma_1}+\dots
           +|a_k\gamma_k(\gamma_k-1)(\gamma_k-2)|M^{\gamma_k}}\right)^{\frac{1}{k-1}}\nonumber\\
\ll & \,\, M \left(\Delta^*_0M^3R^{-1}\right)^{\frac{1}{k-1}}
           \ll M^{1+\frac{3}{k-1}}(\Delta^*_0)^{\frac{1}{k-1}}R^{-\frac{1}{k-1}}.
\end{align}
From (\ref{3-derivative}) of Lemma \ref{2-3-derivative} and (\ref{etimate-mathcal(I)0}), by taking $\Delta^*_0=M^{-3}R^{3/(k+2)}$, we deduce that
\begin{align*}
           S_k(M;\gamma_1,\dots,\gamma_k)
  = & \,\, \sum_{m\in \mathcal{I}_0}e\left(f_k(m)\right)+\sum_{1\leqslant t\leqslant T^*}\sum_{m\in
           \mathcal{I}_{t^*}}e\left(f_k(m)\right)
                 \nonumber \\
\ll & \,\, M^{1+\frac{3}{k-1}}(\Delta^*_0)^{\frac{1}{k-1}}R^{-\frac{1}{k-1}}+\sum_{1\leqslant t\leqslant T^*}
           \left(M(2^{t^*}\Delta^*_0)^{1/6}+(2^{t^*}\Delta^*_0)^{-1/3}\right)
                 \nonumber \\
\ll & \,\, M^{1+\frac{3}{k-1}}(\Delta^*_0)^{\frac{1}{k-1}}R^{-\frac{1}{k-1}}
           +M^{1/2}R^{1/6}+(\Delta^*_0)^{-1/3}
                 \nonumber \\
\ll & \,\, M^{1/2}R^{1/6}+MR^{-\frac{1}{k+2}},
\end{align*}
which gives the estimate (\ref{prop-Sk(M-gamma-alpha)-3}) of Proposition \ref{prop-Sk(M-gamma-alpha)}.
\end{proof}

In the rest of this section, we shall estimate an exponential sum over prime by using Lemma \ref{H-B-Identity} and Proposition \ref{prop-Sk(M-gamma-alpha)}. For $k\geqslant 3$, let $1/2<\gamma_k<\dots<\gamma_1<1,\,\, m\sim M,\,n\sim N,\,mn\sim X$, and $h_j\,(1\leqslant j\leqslant k)$ be integers satisfying $1\leqslant|h_j|\leqslant X^{1-\gamma_j+\delta+\varepsilon}$. Define
\begin{align*}
& S_I(M,N):=\sum_{M<m\leqslant2M}a(m)\sum_{N<n\leqslant2N}e\left(\alpha mn+h_1(mn)^{\gamma_1}+\dots+h_k(mn)^{\gamma_k}\right),\\
& S_{II}(M,N):=\sum_{M<m\leqslant2M}a(m)\sum_{N<n\leqslant2N}b(n)e\left(\alpha mn+h_1(mn)^{\gamma_1}+\dots+h_k(mn)^{\gamma_k}\right),
\end{align*}
where $a(m)$ and $b(n)$ are complex numbers satisfying $a(m)\ll1,\,b(n)\ll1$. For convenience, we write
$\mathscr{R}=|h_1|X^{\gamma_1}+\dots+|h_k|X^{\gamma_k}$. Trivially, there holds $X^{\gamma_1}\ll\mathscr{R}\ll X^{1+\varepsilon}$. Define
\begin{align*}
\mathcal{X}_k=
\begin{cases}
X^{\frac{1}{48}-\frac{1}{2}}\mathscr{R}, & k=3,\\
X^{\frac{\varpi_k}{2}-\frac{1}{2}}\mathscr{R}, & k\geqslant4,
\end{cases}
\qquad
\mathcal{X}^*_k=
\begin{cases}
X^{\frac{1}{2}+\frac{49}{144}}\mathscr{R}^{-1/3}, & k=3,\\
X^{\frac{1}{2}+\varpi_k},     & k\geqslant4.
\end{cases}
\end{align*}

\begin{lemma}\label{SII(M,N)}
Suppose that $a(m)\ll1,\,b(n)\ll1$. For $k\geqslant3$, let $\sigma_k=k-(\gamma_1+\dots+\gamma_k)$. Then, for
$X^{1/6}\ll N\ll \mathcal{X}_k$, we have
\begin{equation*}
S_{II}(M,N)\ll X^{1-\sigma_k-(k+1)\delta-(k+3)\varepsilon}.
\end{equation*}
\end{lemma}
\begin{proof}
Let $Q=X^{2\sigma_k+2(k+1)\delta+2(k+3)\varepsilon}=o(N)$. Then, by Cauchy's inequality and Lemma \ref{Weyl-Inequality}, we obtain
\begin{align*}
           |S_{II}(M,N)|^2
\ll & \,\, M\sum_{M<m\leqslant2M}\left|\sum_{N<n\leqslant2N}b_ne(\alpha mn+h_1(mn)^{\gamma_1}
           +\dots+h_k(mn)^{\gamma_k})\right|^2
                  \nonumber \\
\ll & \,\, \frac{M^2N^2}{Q}+\frac{MN}{Q}\sum_{1\leqslant q\leqslant Q}
           \sum_{N<n\leqslant2N-q}\left|\sum_{M<m\leqslant2M}e\left(f(m,n;q)\right)\right|
                  \nonumber \\
\ll & \,\, X^{2-2\sigma_k-2(k+1)\delta-2(k+3)\varepsilon},
\end{align*}
where
\begin{align}
           f(m,n;q)
  = & \,\, \alpha qm+h_1m^{\gamma_1}\Delta(n,q;\gamma_1)
            +\dots+h_km^{\gamma_k}\Delta(n,q;\gamma_k),   \label{f(m,n;q)}   \\
           \Delta(n,q;\gamma_j)
  = & \,\, (n+q)^{\gamma_j}-n^{\gamma_j}=\gamma_jqn^{\gamma_j-1}
            +O\left(q^2n^{\gamma_j-2}\right),\qquad(1\leqslant j\leqslant k).  \label{Delta(n,q;gamma-3)}
\end{align}
Hence, it suffices to show that
\begin{equation*}
\mathscr{S}^*_q:=\sum_{1\leqslant q\leqslant Q}
\sum_{N<n\leqslant2N-q}\left|\sum_{M<m\leqslant2M}e\left(f(m,n;q)\right)\right|\ll X.
\end{equation*}
Taking $(\alpha,a_1,\dots,a_k)$=$(\alpha q,h_1\Delta(n,q;\gamma_1),\dots,h_k\Delta(n,q;\gamma_k))$ in Proposition \ref{prop-Sk(M-gamma-alpha)}, it is easy to see that $R=|h_1|qN^{-1}M^{\gamma_1}N^{\gamma_1}+\dots+|h_k|qN^{-1}M^{\gamma_k}N^{\gamma_k}=qN^{-1}\mathscr{R}$. By (\ref{prop-Sk(M-gamma-alpha)-2}) of Proposition \ref{prop-Sk(M-gamma-alpha)}, we get
\begin{equation*}
\sum_{M<m\leqslant2M}e\left(f(m,n;q)\right)\ll \bigg(\frac{q\mathscr{R}}{N}\bigg)^{1/2}+M\bigg(\frac{q\mathscr{R}}{N}\bigg)^{-\frac{1}{k+1}}.
\end{equation*}
Thus, we derive that
\begin{align*}
\mathscr{S}^*_q
 & \ll\sum_{1\leqslant q\leqslant Q}\sum_{N<n\leqslant2N-q}\left(\left(\frac{q\mathscr{R}}{N}\right)^{1/2}
 +M\left(\frac{q\mathscr{R}}{N}\right)^{-\frac{1}{k+1}}\right)\\
 & \ll N^{1/2}\mathscr{R}^{1/2}Q^{3/2}+XN^{\frac{1}{k+1}}\mathscr{R}^{-\frac{1}{k+1}}Q^{\frac{k}{k+1}}\ll X.
\end{align*}
This completes the proof of Lemma \ref{SII(M,N)}.
\end{proof}

\begin{lemma}\label{SI(M,N)}
Suppose that $a(m)\ll1$. For $k\geqslant3$, let $\sigma_k=k-(\gamma_1+\dots+\gamma_k)$. Then, for
$M\ll \mathcal{X}^*_k$, we have
\begin{equation*}
S_{I}(M,N)\ll X^{1-\sigma_k-(k+1)\delta-(k+3)\varepsilon}.
\end{equation*}
\end{lemma}
\begin{proof}
Taking $(\alpha,a_1,\dots,a_k)$=$(\alpha m,h_1M^{\gamma_1},\dots,h_kM^{\gamma_k})$ in Proposition \ref{prop-Sk(M-gamma-alpha)}, it is easy to see that $R=|h_1|M^{\gamma_1}N^{\gamma_1}+\dots+|h_k|M^{\gamma_k}N^{\gamma_k}=\mathscr{R}$.
Therefore, by (\ref{prop-Sk(M-gamma-alpha)-3}) of Proposition \ref{prop-Sk(M-gamma-alpha)}, we get
\begin{align*}
           S_I(M,N)
\ll & \,\, \sum_{M<m\leqslant2M}\left|\sum_{N<n\leqslant2N}e\left(\alpha mn+h_1m^{\gamma_1}n^{\gamma_1}
           +\dots+h_km^{\gamma_k}n^{\gamma_k}\right)\right|   \\
\ll & \,\, \sum_{M<m\leqslant2M}\left(N^{1/2}\mathscr{R}^{1/6}+N\mathscr{R}^{-\frac{1}{k+2}}\right)
           \ll MN^{1/2}\mathscr{R}^{1/6}+MN\mathscr{R}^{-\frac{1}{k+2}}  \\
\ll & \,\,  X^{1-\sigma_k-(k+1)\delta-(k+3)\varepsilon}.
\end{align*}
This completes the proof of Lemma \ref{SI(M,N)}.
\end{proof}

\begin{lemma}\label{Apply-H-B-identity}
Suppose that $N/2\ll X_1\ll N$, $k\geqslant3$ is a fixed integer, $\sigma_k=k-(\gamma_1+\dots+\gamma_k)$.
Then we have
\begin{equation*}
\mathscr{T}_k(X):=\sum_{X<n\leqslant X_1}\Lambda(n)e\left(\alpha n+h_1n^{\gamma_1}+\dots+h_kn^{\gamma_k}\right)\ll X^{1-\sigma_k-(k+1)\delta-(k+2)\varepsilon}.
\end{equation*}
\end{lemma}
\begin{proof}
By Lemma \ref{H-B-Identity} with $\nu=4$, one can see that $\mathscr{T}_k(X)$ can be written as linear combination of $O(\log^8X)$ sums, each of which is of form
\begin{align}\label{S*(n1-n-8)}
\mathscr{T}^*(X)& :=\sum_{n_1\sim N_1}\dots\sum_{n_8\sim N_8}(\log n_1)\mu(n_5)\mu(n_6)\mu(n_7)\mu(n_8)\nonumber\\
& \qquad\times e\left(\alpha (n_1\cdots n_8)+h_1(n_1\cdots n_8)^{\gamma_1}+\dots+h_k(n_1\cdots n_8)^{\gamma_k}\right),
\end{align}
where $X\ll N_1\cdots N_8\ll X; N_l\leqslant(2X)^{1/4},l=5,6,7,8$ and some $n_l$ may only take value $1$. Therefore, it suffices to give upper bound estimate for each $\mathscr{T}^*(X)$ defined as in (\ref{S*(n1-n-8)}). Next, we will consider three cases.

\noindent
\textbf{Case 1}. If there exists an $N_j$ such that $N_j\gg \mathcal{X}^*_k>X^{1/4}$, we must have $1\leqslant j\leqslant4$. Without loss of generality, we postulate that $N_j=N_1$ and take $m=n_2n_3\cdots n_8,\,n=n_1$. Trivially, there holds $m\ll \mathcal{X}^*_k$. Set
\begin{equation*}
a(m):=\sum_{m=n_2n_3\cdots n_8}\mu(n_5)\cdots\mu(n_8)\ll d_7(m).
\end{equation*}
Then $\mathscr{T}^*(X)$ is a sum of the form $S_{I}(M,N)$. By Lemma \ref{SI(M,N)}, we get
\begin{equation*}
X^{-\varepsilon}\mathscr{T}^*(X)\ll X^{1-\sigma_k-(k+1)\delta-(k+4)\varepsilon}.
\end{equation*}

\noindent
\textbf{Case 2.} If there exists an $N_j$ such that $X^{1/6}\ll N_j\ll\mathcal{X}_k$, then we take
\begin{equation*}
n=n_j,\quad m=\prod_{l\neq j}n_l,\quad N=N_j,\quad M=\prod_{l\neq j}N_l.
\end{equation*}
Thus, $\mathscr{T}^*(X)$ is a sum of the form $S_{II}(M,N)$ with $X^{1/6}\ll N\ll \mathcal{X}_k$. By Lemma \ref{SII(M,N)}, we derive that
\begin{equation*}
X^{-\varepsilon}\mathscr{T}^*(X)\ll X^{1-\sigma_k-(k+1)\delta-(k+4)\varepsilon}.
\end{equation*}

\noindent
\textbf{Case 3.} If $N_j< X^{1/6}\,(j=1,2,\dots,8)$, without loss of generality, we assume that $N_1\geqslant N_2\geqslant\cdots \geqslant N_8$. Let $\ell$ be an integer such that
\begin{equation*}
N_1\cdots N_{\ell-1}< X^{1/6},\qquad N_1\cdots N_{\ell-1}N_\ell\geqslant X^{1/6}.
\end{equation*}
It is easy to check that $2\leqslant\ell\leqslant7$. Then we have
\begin{equation*}
X^{1/6}<N_1\cdots N_\ell=N_1\cdots N_{\ell-1}N_\ell\leqslant X^{1/3}\leqslant \mathcal{X}_k.
\end{equation*}
In this case, we take
\begin{equation*}
m=\prod_{j=\ell+1}^8n_j,\quad n=\prod_{j=1}^\ell n_j,\quad M=\prod_{j=\ell+1}^8N_j,\quad N=\prod_{j=1}^\ell N_j.
\end{equation*}
Then $\mathscr{T}^*(X)$ is a sum of the form $S_{II}(M,N)$. By Lemma \ref{SII(M,N)}, we have
\begin{equation*}
X^{-\varepsilon}\mathscr{T}^*(X)\ll X^{1-\sigma_k-(k+1)\delta-(k+4)\varepsilon}.
\end{equation*}
Combining the above three cases, we derive that
\begin{equation*}
\mathscr{T}_k(X)\ll\mathscr{T}^*(X)\cdot\log^8X\ll X^{1-\sigma_k-(k+1)\delta-(k+2)\varepsilon}.
\end{equation*}
This completes the proof of Lemma \ref{Apply-H-B-identity}.
\end{proof}

\section{Proof of Theorem \ref{Theorem 1.1}}

\begin{proposition}\label{estimate-k-psi}
Suppose that $\sigma_k=k-(\gamma_1+\dots+\gamma_k)$. For $k\geqslant3$, let $1/2<\gamma_k<\dots<\gamma_1<1$, $0\leqslant\delta<1/2$, and $\Omega(k),\Omega_*(k)$ be constants which satisfy
\begin{equation*}
\Omega(k)\cdot\sigma_k+\Omega_*(k)\cdot\delta<1-\varpi_k.
\end{equation*}
Then, uniformly for $\alpha\in[0,1]$, we have
\begin{equation*}
\sum_{p\leqslant N}p^{\sigma_k}e(\alpha p)\prod_{j=1}^k\left(\psi(-(p+1)^{\gamma_j})-\psi(-p^{\gamma_j})\right)\ll N^{1-\delta-\varepsilon},
\end{equation*}
where the implied constant may depend on $\gamma_j\,(1\leqslant j\leqslant k)$ and $\delta$.
In particular, one has
\begin{align*}
&\Omega(3)=12,\quad\Omega_*(3)=52\quad \Longrightarrow\quad 12\cdot\sigma_3+52\cdot\delta<1-1/24;\\
&\Omega(k)=4k,\quad \Omega_*(k)=4k(k+1)\quad \Longrightarrow\quad 4k\cdot\sigma_k+4k(k+1)\cdot\delta<1-\varpi_k,
\end{align*}
where $\varpi_k$ is defined as in (\ref{varpi-k}).
\end{proposition}
\begin{proof}
Obviously, it is sufficient to show that, for $N^{1/2}\ll X\ll N$, there holds
\begin{equation*}
S_k(X):=\sum_{X<n\leqslant2X}\Lambda(n)e(\alpha n)n^{\sigma_k}\prod_{j=1}^k\left(\psi(-(n+1)^{\gamma_j})-\psi(-n^{\gamma_j})\right)\ll X^{1-\delta-\varepsilon}.
\end{equation*}
Let $H_j=X^{1-\gamma_j+\delta+\varepsilon}\,(j=1,2,\dots,k)$. By Lemma  ~\ref{Finite-Fourier-expansion}, we know that
\begin{equation}\label{psi-psi-3}
\psi(-(n+1)^{\gamma_j})-\psi(-n^{\gamma_j})=M(n,\gamma_j)+E(n,\gamma_j),
\end{equation}
where
\begin{align}
& M(n,\gamma_j)=-\sum_{0<|h_j|\leqslant H_j}\frac{e\left(-h_j(n+1)^{\gamma_j}\right)-e(-h_jn^{\gamma_j})}{2\pi ih_j},\label{M(n,gamma_j)}\\
& E(n,\gamma_j)=O\left(\max_{v_j\in[0,1]}\min\left(1,\frac{1}{H_j\|(n+v_j)^{\gamma_j}\|}\right)\right).
\label{E(n,gamma_j)}
\end{align}
For fixed $j\in\{1,2,\dots,k\}$, let
\begin{align*}
& \phi_{\gamma_j,n}(u)=u^{-1}\left(e(u((n+1)^{\gamma_j}-n^{\gamma_j}))-1\right),\\
& S_{\gamma_j,n}(t)=\sum_{1\leqslant h_j\leqslant t}e(h_jn^{\gamma_j}),\qquad  1<t\leqslant H_j.
\end{align*}
It is easy to check that
\begin{equation*}
 \phi_{\gamma_j,n}(u)\ll n^{\gamma_j-1},\quad\frac{\partial\phi_{\gamma_j,n}(u)}{\partial u}\ll u^{-1}n^{\gamma_j-1},\quad S_{\gamma_j,n}(t)\ll\min\left(t,\frac{1}{\|n^{\gamma_j}\|}\right).
\end{equation*}
By partial summation, we derive that
\begin{align}\label{M(n,gamma_j)-Esti}
           |M(n,\gamma_j)|
\ll & \,\, \bigg|\sum_{1\leqslant h_j\leqslant H_j}e(h_jn^{\gamma_j})\phi_{\gamma_j,n}(h_j)\bigg|
           =\left|\int_{1}^{H_j}\phi_{\gamma_j,n}(t)\mathrm{d}S_{\gamma_j,n}(t)\right|
                 \nonumber\\
\ll & \,\, \left|\phi_{\gamma_j,n}(H_j)\right|\left|S_{\gamma_j,n}(H_j)\right|+\int_{1}^{H_j}
           \left|S_{\gamma_j,n}(t)\right|
           \left|\frac{\partial\phi_{\gamma_j,n}(t)}{\partial t}\right|\mathrm{d}t
                 \nonumber\\
\ll & \,\, H_jX^{\gamma_j-1}\log X\cdot\min\left(1,\frac{1}{H_j\|n^{\gamma_j}\|}\right).
\end{align}
Form (\ref{psi-psi-3}), (\ref{E(n,gamma_j)}) and (\ref{M(n,gamma_j)-Esti}), we get
\begin{align}\label{S_k(X)}
S_k(X)
& =\sum_{X<n\leqslant2X}\Lambda(n)e(\alpha n)n^{\sigma_k}\prod_{j=1}^k\left(M(n,\gamma_j)+E(n,\gamma_j)\right)\nonumber\\
& =\sum_{X<n\leqslant2X}\Lambda(n)e(\alpha n)n^{\sigma_k}\prod_{j=1}^kM(n,\gamma_j)+O(\mathscr{S}_k(X)),
\end{align}
where
\begin{align*}
 & \mathscr{S}_k(X)\ll X^{\sigma_k+(k-1)\delta+k+3/2(k-1)\varepsilon}
 \mathscr{S}^*(X;\gamma_1,\dots,\gamma_k),\\
 & \mathscr{S}^*(X;\gamma_1,\dots,\gamma_k)=\sup_{(v_1,\dots,v_k)\in[0,1]^k}\sum_{X<n\leqslant2X}
 \prod_{j=1}^k\min\left(1,\frac{1}{\|H_j(n+v_j)^{\gamma_j}\|}\right).
\end{align*}
By Lemma \ref{Zhai-1999-prop-4} with $s=k$, we obtain
\begin{equation}\label{mathscr{S}_k(X)}
\mathscr{S}_k(X)\ll X^{\sigma_k+(k-1)\delta+\frac{3}{2}(k-1)\varepsilon}\left(X(H_1\cdots H_k)^{-1}\log^k X+X^{\frac{k}{k+1}}\log^k X\right)\ll X^{1-\delta-\varepsilon}.
\end{equation}
Next, we shall give the upper bound estimate of the first term on the right--hand of (\ref{S_k(X)}). Define
\begin{equation}\label{Psi(n)}
\Psi_{\gamma_j}(n)=e\left(h_j((n+1)^{\gamma_j}-n^{\gamma_j})\right)-1,
\qquad \Psi(n)=\prod_{j=1}^k\Psi_{\gamma_j}(n).
\end{equation}
It is easy to see that
\begin{align}
 & \Psi_{\gamma_j}(n)\ll|h_j|n^{\gamma_j-1},\quad \frac{\partial\Psi_{\gamma_j}(n)}{\partial n}\ll|h_j|n^{\gamma_j-2},\nonumber\\
 & \Psi(n)\ll |h_1\cdots h_k|n^{\gamma_1+\dots+\gamma_k-k},\quad
 \frac{\partial\Psi(n)}{\partial n}\ll |h_1\cdots h_k|n^{\gamma_1+\dots+\gamma_k-k-1}.\label{Psi(n)-derivative}
\end{align}
It follows from partial summation, (\ref{M(n,gamma_j)}), (\ref{Psi(n)}), (\ref{Psi(n)-derivative}) and Lemma \ref{Apply-H-B-identity} that the first term on the right--hand side of (\ref{S_k(X)}) can be estimated as
\begin{align}\label{mathcal{S}_k(X)}
           \mathcal{S}_k(X)
 := & \,\,  \sum_{X<n\leqslant2X}\Lambda(n)e(\alpha n)n^{\sigma_k}\prod_{j=1}^kM(n,\gamma_j)
                  \nonumber\\
\ll & \,\, \sum_{1\leqslant h_1\leqslant H_1}\cdots\sum_{1\leqslant h_k\leqslant H_k}\frac{1}{|h_1\cdots h_k|}
           \left|\sum_{X<n\leqslant2X}\Lambda(n)e(\alpha n)
           n^{\sigma_k}\prod_{j=1}^ke(h_jn^{\gamma_j})\Psi_{\gamma_j}(n)\right|
                  \nonumber\\
\ll & \,\, \sum_{1\leqslant h_1\leqslant H_1}\cdots\sum_{1\leqslant h_k\leqslant H_k}\frac{1}{|h_1\cdots h_k|}
           \left|\sum_{X<n\leqslant2X}\Lambda(n)e(\alpha n)n^{\sigma_k}e\bigg(\alpha n+\sum_{j=1}^kh_jn^{\gamma_j}\bigg)\Psi(n)\right|
                   \nonumber\\
\ll & \,\, \sum_{1\leqslant h_j\leqslant H_1}\cdots\sum_{1\leqslant h_j\leqslant H_k}\frac{1}{|h_1\cdots h_k|}
           \left|\int_{X}^{2X}\Psi(t)t^{\sigma_k}\mathrm{d}\bigg(\sum_{X<n\leqslant t}\Lambda(n)
           e\bigg(\alpha n+\sum_{j=1}^kh_jn^{\gamma_j}\bigg)\bigg)\right|
                   \nonumber\\
\ll & \,\, \log X\max_{\substack{1\leqslant h_i\leqslant H_i\\1\leqslant i\leqslant k}}
           \max_{X<t\leqslant2X}\sum_{1\leqslant h_1\leqslant H_1}\cdots\sum_{1\leqslant h_k\leqslant H_k}
           \left|\sum_{X<n\leqslant t}\Lambda(n)e\bigg(\alpha n+\sum_{j=1}^kh_jn^{\gamma_j}\bigg)\right|
                   \nonumber\\
\ll & \,\, X^{\sigma_k+k\delta+k\varepsilon}\cdot X^{1-\sigma_k-(k+1)\delta-(k+2)\varepsilon}\log X
           \ll X^{1-\delta-\varepsilon}.
\end{align}
Combining (\ref{S_k(X)}), (\ref{mathscr{S}_k(X)}) and (\ref{mathcal{S}_k(X)}), we derive the desired result.
\end{proof}

\noindent
\textbf{\textit{Proof of Theorem \ref{Theorem 1.1}}}. Let $k\geqslant3$ be a fixed integer. Suppose that $1/2<\gamma^{(i)}_k<\dots<\gamma^{(i)}_1<1\,(i=1,2,3)$ are fixed real constants. Bearing in mind (\ref{define-symbol}), we define
\begin{equation*}
G(\alpha):=\sum_{p\leqslant N}e(\alpha p).
\end{equation*}
By the orthogonality, we obtain
\begin{equation*}
\mathscr{R}(N):=\sum_{N=p_1+p_2+p_3}1=\int_{0}^1G^3(\alpha)e(-N\alpha)\mathrm{d}\alpha.
\end{equation*}
Define
\begin{equation*}
W_i(\alpha):=\mathcal{C}^{(i)}_k\sum_{\substack{p\leqslant N\\p=\lfloor n^{1/\gamma^{(i)}_1}_1\rfloor=\cdots=\lfloor n^{1/\gamma^{(i)}_k}_k\rfloor}}p^{\sigma^{(i)}_k}e(\alpha p).
\end{equation*}
Then one has
\begin{equation*}
\mathscr{W}(N) :=\mathcal{C}^{(1)}_k\mathcal{C}^{(2)}_k\mathcal{C}^{(3)}_k\sum_{\substack{N=p_1+p_2+p_3
\\p=\lfloor n^{1/\gamma^{(i)}_1}_1\rfloor=\cdots=\lfloor n^{1/\gamma^{(i)}_k}_k\rfloor}}\bigg(\prod_{i=1}^3p^{\sigma^{(i)}_k}\bigg) =\int_{0}^1\bigg(\prod_{i=1}^3W_i(\alpha)\bigg)e(-N\alpha)\mathrm{d}\alpha.
\end{equation*}
According to (\ref{1.2}), it suffices to show that
\begin{equation*}
\mathscr{W}(N)-\mathscr{R}(N)=:\mathcal{E}\ll N^{2-\varepsilon}.
\end{equation*}
It is easy to check that
\begin{equation}\label{Elementary-Identity}
W_1W_2W_3-G^3=(W_1-G)W_2W_3+GW_3(W_2-G)+G^2(W_3-G).
\end{equation}
By (\ref{Elementary-Identity}), one has
\begin{align}\label{Error}
\mathcal{E} & \ll
\left(\sup_{\alpha}\left|W_1-G\right|\right)\int_{0}^1\left|W_2W_3\right|\mathrm{d}\alpha
+\left(\sup_{\alpha}\left|W_2-G\right|\right)\int_{0}^1\left|GW_3\right|\mathrm{d}\alpha\nonumber\\
& \quad +\left(\sup_{\alpha}\left|W_3-G\right|\right)\int_{0}^1\left|G\right|^2\mathrm{d}\alpha.
\end{align}
By Lemma \ref{Zhai-1999-Th1}, we know that
\begin{align}\label{Wi^2}
           \int_{0}^1|W_i(\alpha)|^2\mathrm{d}\alpha
 = & \,\,  \big(\mathcal{C}^{(i)}_k\big)^2\int_{0}^1\sum_{\substack{p\leqslant N\\p=\lfloor
           n^{1/\gamma^{(i)}_1}_1\rfloor=\cdots=\lfloor n^{1/\gamma^{(i)}_k}_k\rfloor}}
           \sum_{\substack{p'\leqslant N\\p'=\lfloor n^{1/\gamma^{(i)}_1}_1\rfloor=\cdots=\lfloor n^{1/\gamma^{(i)}_k}_k\rfloor}}
           \left(pp'\right)^{\sigma^{(i)}_k}e(\alpha(p-p'))\mathrm{d}\alpha
                 \nonumber\\
\ll & \,\, \sum_{\substack{p\leqslant N\\ p=\lfloor n^{1/\gamma^{(i)}_1}_1\rfloor=\cdots=\lfloor
           n^{1/\gamma^{(i)}_k}_k\rfloor}}p^{2\sigma^{(i)}_k}\ll N^{2\sigma^{(i)}_k}
           \sum_{\substack{p\leqslant N\\p=\lfloor n^{1/\gamma^{(i)}_1}_1\rfloor=\cdots=\lfloor n^{1/\gamma^{(i)}_k}_k\rfloor}}1
                 \nonumber\\
\ll & \,\, N^{2\sigma^{(i)}_k}_k\cdot\frac{N^{1-\sigma^{(i)}_k}}{\log N}\ll\frac{N^{\sigma^{(i)}_k+1}}{\log N}.
\end{align}
Therefore, the three integrals in (\ref{Error}) can be estimated by prime number theorem, Cauchy's inequality and (\ref{Wi^2}). More precisely, we derived that
\begin{align*}
 & \int_{0}^1|G|^2\mathrm{d}\alpha=\int_{0}^1\sum_{p\leqslant N}e(\alpha p)\sum_{p'\leqslant N}e(-\alpha p')\mathrm{d}\alpha\ll\frac{N}{\log N},\\
 & \int_{0}^{1}\left|GW_3\right|\mathrm{d}\alpha\leqslant\left(\int_{0}^1|G|^2\mathrm{d}\alpha\right)^{1/2} \left(\int_{0}^1|W_3|^2\mathrm{d}\alpha\right)^{1/2}\ll\frac{N^{1+\frac{1}{2}\sigma^{(3)}_k}}{\log N},\\
 & \int_{0}^{1}\left|W_2W_3\right|\mathrm{d}\alpha\leqslant\left(\int_{0}^1|W_2|^2\mathrm{d}\alpha\right)^{1/2} \left(\int_{0}^1|W_3|^2\mathrm{d}\alpha\right)^{1/2}\ll\frac{N^{1+\frac{1}{2}
\left(\sigma^{(2)}_k+\sigma^{(3)}_k\right)}}{\log N}.
\end{align*}
Combining the above three estimates and (\ref{Error}), it is sufficient to show that
\begin{equation*}
\sup_{\alpha\in(0,1)}\left|W_i-G\right|\ll N^{1-\delta_i-\varepsilon}    \qquad (i=1,2,3),
\end{equation*}
where
\begin{equation*}
\delta_3=0,\qquad\delta_2=\frac{1}{2}\sigma^{(3)}_k,\qquad \delta_1=\frac{1}{2}\Big(\sigma^{(2)}_k+\sigma^{(3)}_k\Big).
\end{equation*}
For $1/2<\gamma<1$, it is easy to see that
\begin{equation}\label{Character-p=[c]}
\lfloor-p^{\gamma}\rfloor-\lfloor-(p+1)^{\gamma}\rfloor=
\begin{cases}
     1,  & \textrm{if $p=\lfloor n^{1/\gamma}\rfloor$},\\
     0, & \textrm{otherwise}.
   \end{cases}
\end{equation}
For convenience, we write
\begin{equation}\label{D(p,gamma)+E(p,gamma)}
\lfloor-p^{\gamma}\rfloor-\lfloor-(p+1)^{\gamma}\rfloor=\mathscr{D}(p,\gamma)+\mathscr{E}(p,\gamma),
\end{equation}
where
\begin{align}
& \mathscr{D}(p,\gamma)=(p+1)^{\gamma}-p^{\gamma}=\gamma p^{\gamma-1}+O(p^{\gamma-2}),\label{D(p,gamma)-asymptotic}\\
& \mathscr{E}(p,\gamma)=\psi(-(p+1)^{\gamma})-\psi(-p^{\gamma}).\label{E(p,gamma)-psi-psi}
\end{align}
Let $K=\{1,2,\dots,k\},\,k\geqslant3$. From (\ref{Character-p=[c]}) and (\ref{D(p,gamma)+E(p,gamma)}), we deduce that
\begin{align*}
           W(\alpha)
:= & \,\, \mathcal{C}_k\sum_{\substack{p\leqslant N\\p=\lfloor n^{1/\gamma_1}_1\rfloor=\cdots=\lfloor
          n^{1/\gamma_k}_k\rfloor}}p^{\sigma_k}e(\alpha p)
           =\mathcal{C}_k\sum_{p\leqslant N}
           p^{\sigma_k}e(\alpha p)\prod_{j=1}^k\big(\mathscr{D}(p,\gamma_j)+\mathscr{E}(p,\gamma_j)\big)
               \nonumber \\
= & \,\, \mathcal{C}_k\sum_{p\leqslant N}p^{\sigma_k}e(\alpha p)
         \left(\prod_{j=1}^k\mathscr{D}(p,\gamma_j)+\sum_{d=1}^k\sum_{\{i_1,\dots,i_d\}\subseteq K}
         \mathcal{W}(N;p,\gamma_{i_1},\dots,\gamma_{i_d})\right)
               \nonumber \\
= & \,\, \mathfrak{W}_1(\alpha)+\mathfrak{W}_2(\alpha),
\end{align*}
where
\begin{align}
 & \mathfrak{W}_1(\alpha)=\mathcal{C}_k\sum_{p\leqslant N}p^{\sigma_k}e(\alpha p)
   \prod_{j=1}^k\mathscr{D}(p,\gamma_j),
               \nonumber\\
 & \mathfrak{W}_2(\alpha)=\mathcal{C}_k\sum_{p\leqslant N}p^{\sigma_k}e(\alpha p)
    \sum_{d=1}^k\sum_{\{i_1,\dots,i_d\}\subseteq K}
    \mathcal{W}(N;p,\gamma_{i_1},\dots,\gamma_{i_d}),\label{W2(alpha)}   \\
  &
 \mathcal{W}(N;p,\gamma_{i_1},\dots,\gamma_{i_d})=\mathscr{E}(p,\gamma_{i_1})\cdots\mathscr{E}(p,\gamma_{i_d}) \prod_{\substack{j\neq i_l\\l=1,\dots,d}}\mathscr{D}(p,\gamma_j).   \nonumber
\end{align}
By (\ref{D(p,gamma)-asymptotic}), we get
\begin{align*}
\mathfrak{W}_1(\alpha)
 & =\mathcal{C}_k\sum_{p\leqslant N}p^{\sigma_k}e(\alpha p)\prod_{j=1}^k\left((p+1)^{\gamma_j}-p^{\gamma_j}\right)
   =\mathcal{C}_k\sum_{p\leqslant N}p^{\sigma_k}e(\alpha p)
   \prod_{j=1}^k\left(\gamma_jp^{\gamma_j-1}+O(p^{\gamma_j-2})\right)\\
 & =\sum_{p\leqslant N}e(\alpha p)+O\left(\bigg|\sum_{p\leqslant N}e(\alpha p)
   p^{-1}\bigg|\right)=G(\alpha)+O(\log\log N).
\end{align*}
Thus, it suffices to show that
\begin{equation*}
\mathfrak{W}_2(\alpha)\ll N^{1-\delta-\varepsilon}.
\end{equation*}
If $d=1$ in (\ref{W2(alpha)}), then $\{i_1\}\subseteq K$. It follows from (\ref{D(p,gamma)-asymptotic}) and Lemma \ref{Balog-Friedlander-Th4} that
\begin{align}\label{mathfrak{W}_{21}}
\mathfrak{W}_{21}(\alpha)
:= & \,\, \mathcal{C}_k\sum_{\{i_1\}\subseteq K}\sum_{p\leqslant N}e(\alpha p)
          p^{\sigma_k}\mathscr{E}(p,\gamma_{i_1})\prod_{\substack{j\neq i_{1}\nonumber\\ \gamma_j\subseteq K}}
          \mathscr{D}(p,\gamma_j)
                 \nonumber \\
= & \,\, \mathcal{C}_k\sum_{\{i_1\}\subseteq K}\sum_{p\leqslant N}e(\alpha p)
         p^{\sigma_k}\mathscr{E}(p,\gamma_{i_1})\prod_{\substack{j\neq i_{1}\\ \gamma_j\subseteq K}}
         \left(\gamma_jp^{\gamma_j-1}+O(p^{\gamma_j-2})\right)
                 \nonumber\\
= & \,\, \mathcal{C}_k\sum_{\{i_1\}\subseteq K}\sum_{p\leqslant N}
    p^{1-\gamma_{i_1}}e(\alpha p)\mathscr{E}(p,\gamma_{i_1})+O\Bigg(\bigg|\sum_{p\leqslant N}e(\alpha p)
    p^{-\gamma_{i_1}}\mathscr{E}(p,\gamma_{i_1})\bigg|\Bigg)
\ll N^{1-\delta-\varepsilon}.
\end{align}
If $d=2$ in (\ref{W2(alpha)}), then $\{i_1,i_2\}\subseteq K$. From (\ref{D(p,gamma)-asymptotic}) and Lemma \ref{Li-Zhai-2022}, we derive that
\begin{align}\label{mathfrak{W}_{22}}
          \mathfrak{W}_{22}(\alpha)
:= & \,\, \mathcal{C}_k\sum_{\{i_1,i_2\}\subseteq K}\sum_{p\leqslant N}
          p^{\sigma_k}e(\alpha p)\mathscr{E}(p,\gamma_{i_1})\mathscr{E}(p,\gamma_{i_2})
          \prod_{\substack{j\neq i_1,i_2\\ \gamma_j\subseteq K}}\mathscr{D}(p,\gamma_j)
                \nonumber\\
= & \,\, \mathcal{C}_k\sum_{\{i_1,i_2\}\subseteq K}\sum_{p\leqslant N}
         p^{\sigma_k}e(\alpha p)\mathscr{E}(p,\gamma_{i_1})\mathscr{E}(p,\gamma_{i_2})
         \prod_{\substack{j\neq i_1,i_2\\ \gamma_j\subseteq K}}
         \left(\gamma_jp^{\gamma_j-1}+O(p^{\gamma_j-2})\right)
                \nonumber\\
= & \,\, \mathcal{C}_k\sum_{\{i_1,i_2\}\subseteq K}\sum_{p\leqslant N}
         p^{2-(\gamma_{i_1}+\gamma_{i_2})}e(\alpha p)\mathscr{E}(p,\gamma_{i_1})\mathscr{E}(p,\gamma_{i_2})
                \nonumber\\
  & \quad+O\Bigg(\bigg|\sum_{p\leqslant N}
          p^{1-(\gamma_{i_1}+\gamma_{i_2})}e(\alpha p)\mathscr{E}(p,\gamma_{i_1})
          \mathscr{E}(p,\gamma_{i_2})\bigg|\Bigg)\ll N^{1-\delta-\varepsilon}.
\end{align}
If $3\leqslant d\leqslant k$ in (\ref{W2(alpha)}), by (\ref{D(p,gamma)-asymptotic}) and Proposition \ref{estimate-k-psi}, we obtain
\begin{align}\label{mathfrak{W}_{2d}}
          \mathfrak{W}_{2d}(\alpha)
:= & \,\, \mathcal{C}_k\sum_{\{i_1,\dots,i_d\}\subseteq K}\sum_{p\leqslant N}
          p^{\sigma_k}e(\alpha p)\prod_{l=1}^d\mathscr{E}(p,\gamma_{i_l})
          \prod_{\substack{ j\neq i_l\\ l=1,\dots,d}}\mathscr{D}(p,\gamma_j)
              \nonumber\\
= & \,\,  \mathcal{C}_k\sum_{\{i_1,\dots,i_d\}\subseteq K}\sum_{p\leqslant N}
          p^{\sigma_k}e(\alpha p)\prod_{l=1}^d\mathscr{E}(p,\gamma_{i_l})\prod_{\substack{ j\neq i_l\\ l=1,\dots,d}}\left(\gamma_jp^{\gamma_j-1}+O(p^{\gamma_j-2})\right)
              \nonumber\\
= & \,\,  \mathcal{C}_k\sum_{\{i_1,\dots,i_d\}\subseteq K}\sum_{p\leqslant N}
          p^{d-(\gamma_{i_1}+\dots+\gamma_{i_d})}e(\alpha p)\prod_{l=1}^d\mathscr{E}(p,\gamma_{i_l})
              \nonumber\\
  & \quad+O\Bigg(\bigg|\sum_{p\leqslant N}
          p^{(d-1)-(\gamma_{i_1}+\dots+\gamma_{i_d})}e(\alpha p)
          \prod_{l=1}^d\mathscr{E}(p,\gamma_{i_l})\bigg|\Bigg)
          \ll N^{1-\delta-\varepsilon}.
\end{align}
Combining (\ref{W2(alpha)})--(\ref{mathfrak{W}_{2d}}), we deduce that
\begin{equation*}
\mathfrak{W}_2(\alpha)\ll \mathfrak{W}_{21}(\alpha)+\mathfrak{W}_{22}(\alpha)+\mathfrak{W}_{2d}(\alpha)\ll N^{1-\delta-\varepsilon}.
\end{equation*}
This completes the proof of Theorem \ref{Theorem 1.1}.

\section*{Acknowledgement}
The authors would like to appreciate the referee for his/her patience in refereeing this paper.
This work is supported by Beijing Natural Science Foundation (Grant No. 1242003), and
the National Natural Science Foundation of China (Grant Nos. 11901566, 12001047, 11971476, 12071238).

\end{document}